\documentclass{amsart}

\usepackage{amsmath,amscd}
\usepackage{epsfig}
\usepackage{amssymb}
\usepackage{amsfonts,flafter,epsf}
\usepackage{newlfont}
\usepackage{color}

\usepackage{IMjournal}

\providecommand\newdefinition\newtheorem

\newtheorem{theorem}{Theorem}%
\newtheorem{lemma}[theorem]{Lemma}

\newdefinition{definition}[theorem]{Definition}
\newdefinition{example}[theorem]{Example}
\newdefinition{remark}[theorem]{Remark}

  \newcommand{\ZSet}{{\mathbb{Z}}}

\newcommand{\rk}{\mathop{\mathrm{rk}}\nolimits}
\newcommand{\crk}{\mathop{\mathrm{corank}}\nolimits}%
\newcommand{\HG}{H_1(G)}

\newcommand{\tors}{\mathop{\mathrm{\kern0pt T}}\nolimits} %

\renewcommand{\phi}{\varphi}

\begin{document}
\title{Co-rank and Betti number of a group}
\author{\|Irina |Gelbukh|, Mexico City}

\rec{December 24, 2014}

\begin{abstract}
We study the maximal ranks of a free and a free abelian quotients of a finitely generated group, called co-rank (inner rank, cut number) and the Betti number, respectively.
We show that any combination of these values within obvious constraints is realized for some finitely presented group, which is important for manifold and foliation topology.
\end{abstract}

\def\sep{, }

\begin{keywords}
     co-rank\sep 
     inner rank\sep 
     fundamental group
\end{keywords}
\begin{subjclass}
                 20E05\sep
                 20F34\sep
                 14F35 
\end{subjclass}

We study the relation between the co-rank $\crk(G)$ of a finitely generated group $G$ and its Betti number $b(G)$. 
These values bound the isotropy index $i(G)$ of $G$: $\crk(G)\le i(G)\le b(G)$~\cite{Dimca-Pa-Su,Gelb10,Meln3}.
These notions have important applications in theory of manifolds as the first non-commutative Betti number $b_1'(M)=\crk(\pi_1(M))$, the first Betti number $b_1(M)=b(\pi_1(M))$, and $h(M)=i(\pi_1(M))$, where $\pi_1(M)$ is the fundamental group of the manifold $M$.
For any $n\ge4$, a group is the fundamental group of a smooth closed connected $n$-manifold iff it is finitely presented.
In theory of $2$- and $3$-manifolds, co-rank of the fundamental group coincides with the cut-number, a generalization of the genus for closed surfaces~\cite{Jaco72,Sikora}.
In the theory of foliations of Morse forms, $b_1'(M)$ and $h(M)$ define the topology of the foliation~\cite{Gelb10,Gelb09}, the form's cohomology class~\cite{Gelb14}, and the types of its singularities~\cite{Gelb11}.

\medskip

For a finitely generated abelian group $G=\ZSet^n\oplus T$, where $T$ is finite, its {\em torsion-free rank}, {\em Pr\"ufer rank}, or {\em (first) Betti number}, is $b(G)=\rk(G/T)=n$. 
The latter term extends to finitely generated groups by $b(G)=b(G^{ab})=\rk(G^{ab}/\tors(G^{ab}))$, where $G^{ab}=G/[G,G]$ is the abelianization and $\tors(\cdot)$, the torsion subgroup. In other words:

\begin{definition}\label{def:betti}
The {\em Betti number} $b(G)$ of a finitely generated group $G$ is the maximum rank of a free abelian quotient group of $G$, i.e., the maximum rank of a free abelian group $A$ such that there exists an epimorphism $\phi:G\twoheadrightarrow A$. 
\end{definition}

The term came from geometric group theory, where $G^{ab}$ is called the first homology group $\HG$.
A non-commutative analog of Betti number can be defined as follows:

\begin{definition}[\textnormal{\cite{Jaco72,Leininger}}]\label{def:corank}
The {\em co-rank\/} $\mathop{\mathrm{corank}}(G)$~\cite{Leininger}, 
{\em inner rank\/} $I\!N(G)$~\cite{Jaco72} or  $Ir(G)$~\cite{Lyndon}, 
or {\em first non-commutative Betti number} $b'_1(G)$~\cite{AL} of a finitely generated group $G$ is the maximum rank of a free quotient group of $G$, i.e., the maximum rank of a free group $F$ such that there exists an epimorphism $\phi:G\twoheadrightarrow F$.
\end{definition}

The notion of co-rank is also in a way dual to that of rank, which is the minimum rank of a free group allowing an epimorphism onto $G$. In contrast to rank, co-rank is algorithmically computable for finitely presented groups~\cite{Makanin, Razborov}. 

For example, $\crk(\ZSet^n)=1$, while $b(\ZSet^n)=n$. For a finite group $G$, $\crk(G)=b(G)=0$; the same holds for $G=\ZSet_2*\ZSet_2*\ZSet_2$, even though it is infinite and contains $F_2$ and thus free subgroups of all ranks up to countable.
Obviously, for any finitely generated group, $\crk(G)\le b(G)\le\rk G$
and $b(G)\ge1$ implies $\crk(G)\ge1$.
In this paper we show that these are the only constraints between these values:

\begin{theorem}\label{theor:group(k,m)}
Let $0\le c,b,r\in\ZSet$. 
Then there exists a 
finitely generated 
group $G$ with $\crk(G)=c$, $b(G)=b$, and $\rk G=r$ iff 
\begin{align*}
c=b=0\textrm{\quad or\quad}1\le c\le b\le r;
\end{align*} 
the group can be chosen to be finitely presented and, if $b=r$, torsion-free.
\end{theorem}

\begin{lemma}\label{lem:*=x=+}
Let $G_1,G_2$ be finitely generated groups. Then for the Betti number of the free product and of the direct product,
\begin{align*}%
b(G_1*G_2)=b(G_1\times G_2)=b(G_1)+b(G_2).
\end{align*}
\end{lemma}

\begin{proof}
Obviously, $(G_1*G_2)^{ab}=(G_1\times G_2)^{ab}$.
Denote $G=G_1\times G_2$.
Since epimorphisms $G_i\twoheadrightarrow \ZSet^{b(G_i)}$ onto free abelian groups can be extended to an epimorphism of $G_1\times G_2\twoheadrightarrow\ZSet^{b(G_2)}\times\ZSet^{b(G_2)}=\ZSet^{b(G_2)+b(G_2)}$, we have $b(G)\ge b(G_1)+b(G_2)$.

Let us now show that $b(G)\le b(G_1)+b(G_2)$. Consider the natural homomorphisms $\psi_1:G_1\to G_1\times1\subseteq G$,  $\psi_2:G_1\to 1\times G_2\subseteq G$.
Then $\psi_i$ and an epimorphism onto a free abelian group 
$$
G_i\stackrel{\psi_i}\longrightarrow G=G_1\times G_2\twoheadrightarrow A=\ZSet^{b(G)}
$$
induces a homomorphism $\phi_i:G_i\to A$. 
Since $A_i=\phi_i(G_i)\subseteq A$ are free abelian groups, $\rk A_i\le b(G_i)$.
Since $G=\langle \psi_1(G_1),\psi_2(G_2)\rangle$, we have $A=\langle A_1,A_2\rangle$; in particular, $b(G)=\rk A\le\rk A_1+\rk A_2$.
\end{proof}

\begin{proof}[Proof of Theorem~\ref{theor:group(k,m)}]
For $1\le c\le b\le r$, consider 
$G=\ZSet^{b_1}*\dots*\ZSet^{b_c}*\ZSet_2^{r-b}$
such that 
$\sum_{i=1}^c b_i=b$. 
By~\cite[Proposition 6.4]{Lyndon}, $\crk(G_1* G_2)=\crk(G_1)+\crk(G_2)$, so $\crk(G)=\sum_{i=1}^c\crk(\ZSet^{b_i})=c$.
By Lemma~\ref{lem:*=x=+}, $b(G)=\sum_{i=1}^cb(\ZSet^{b_i})=b$, and by Grushko-Neumann theorem, $\rk G=r$.
\end{proof}

{
\small
\renewcommand\MR[1]{#1} %
%\bibliographystyle{amsplain}
%\bibliography{Gelbukh}
% [1] R. A. Gordon: The Integrals of Lebesgue, Denjoy, Perron and Henstock. Graduate Studies in Math., Vol. 4, AMS, Providence, 1994. Zbl 0807.26004, MR1288751
% [2] \v S. Schwabik, I. Vrko\v c: On Kurzweil-Henstock equiintegrable sequences. Math. Bohem. 121 (1996), 189--207. Zbl 0863.26009, MR1400612 

\newcommand\reftitle[1]{#1}
\newcommand\refvolume[1]{#1}

}

{
\small
{\em Authors' addresses}:
{\em Irina Gelbukh}, 
Centro de Investigaci\'on en Computación (CIC),
Instituto Polit\'ecnico Nacional (IPN), 
Av. Juan de Dios B\'atiz, 07738, DF, Mexico City,
Mexico,
e-mail:~\texttt{gelbukh@member.ams.org}.

}

\label{lastpage}
\end{document}